\documentclass[sigconf]{amsart}

\usepackage{amssymb}

\usepackage{graphicx}
\usepackage{subcaption}

\usepackage[backend=biber,
        style=alphabetic,citestyle=alphabetic,
        maxnames=99,
        url=false,isbn=false,
        hyperref=true,backref=true]{biblatex}

\DefineBibliographyStrings{english}{
backrefpage = {$\uparrow$},
backrefpages = {$\uparrow$}
}
\usepackage{tikz}
\usepackage{tikz-3dplot}

\usepackage{mathrsfs}
\usepackage{enumitem}

\renewcommand{\P}{\mathbb{P}}
\newcommand{\SL}{\textbf{SL}}
\newcommand{\FO}{\textbf{FO}}
\newcommand{\OP}{\textbf{OP}}

\newcommand{\Span}{\text{span}}

\newcommand{\Div}{\text{div}}

\newcommand{\A}{\mathscr{A}}
\newcommand{\LL}{\mathscr{L}}
\newcommand{\OO}{\mathcal{O}}

\usepackage{hyperref}
\usepackage[capitalise]{cleveref}

\theoremstyle{definition}
\newtheorem{dfn}{Definition}[section]

\newcounter{letteredthm}
\newtheorem{letteredthm}[letteredthm]{Theorem}
\renewcommand*{\theletteredthm}{\Alph{letteredthm}}
\theoremstyle{plain}
\newtheorem{thm}[dfn]{Theorem}
\newtheorem{prop}[dfn]{Proposition}
\newtheorem{lem}[dfn]{Lemma}

\newtheorem{cor}[dfn]{Corollary}
\theoremstyle{remark}
\newtheorem{zbbackend}[dfn]{Example}

\newenvironment{zb}[1][]{\begin{zbbackend}[#1]}{\end{zbbackend}}
\newtheorem{rmk}[dfn]{Remark}

\newlist{thmenum}{enumerate}{1} 
\setlist[thmenum]{label=(\roman*), ref=\thethm~(\roman*)}
\crefalias{thmenumi}{theorem} 

\newlist{letteredthmenum}{enumerate}{1} 
\setlist[letteredthmenum]{label=(\roman*), ref=\theletteredthm~(\roman*)}
\crefalias{letteredthmenumi}{theorem} 

\newlist{propenum}{enumerate}{1} 
\setlist[propenum]{label=(\roman*), ref=\thethm~(\roman*)}
\crefalias{propenumi}{proposition} 

\newlist{corenum}{enumerate}{1} 
\setlist[corenum]{label=(\roman*), ref=\thethm~(\roman*)}
\crefalias{corenumi}{corollary}

\newcommand{\C}{\mathbb{C}}
\newcommand{\R}{\mathbb{R}}
\newcommand{\Q}{\mathbb{Q}}
\newcommand{\Z}{\mathbb{Z}}

\renewcommand{\P}{\mathbb{P}}

\DeclareMathOperator{\Hom}{Hom}

\DeclareMathOperator{\Spec}{Spec}
\DeclareMathOperator{\Sym}{Sym}

\usepackage{xcolor}
\usepackage{tikz-cd}
\crefalias{prop}{proposition}
\usepackage{color}

\usepackage{array}
\newcolumntype{L}[1]{>{\raggedright\let\newline\\\arraybackslash\hspace{0pt}}m{#1}}

\title{Hyperplane arrangements and compactifications of vector groups}
\author{Colin Crowley}
\address{
University of Wisconsin–Madison
}
\email{cwcrowley@wisc.edu}

\begin{document}

\begin{abstract}
Matroid Schubert varieties have recently played an essential
role in the proof of the Dowling-Wilson conjecture and in
Kazhdan-Lusztig theory for matroids. We study these varieties as
equivariant compactifications of affine spaces, and give necessary and
sufficient conditions to characterize them. We also
generalize the theory to include partial compactifications and
morphisms between them. Our results resemble the
correspondence between toric varieties and polyhedral fans. 
\end{abstract}

\maketitle

\section{Introduction}
A matroid Schubert variety is a singular algebraic variety 
constructed from a hyperplane arrangement, which is not a Schubert 
variety, but which plays the role of
a Schubert variety in
the role it plays in Kazhdan-Lusztig theory for matroids \cite{BHM+20b}.
Let $H_1,\ldots,H_n$ be a collection of linear hyperplanes in a finite 
dimensional complex vector space $V$. Assume that the arrangement is 
\emph{essential}, meaning that the common intersection of the 
hyperplanes is zero.
The associated matroid Schubert variety is the closure of $V$ in 
$(\P^1)^n$ via the embedding
\[
V \subseteq V/H_1 \times \ldots \times V/H_n \subseteq \P^1 \times \ldots \times \P^1.
\]

Matroid Schubert varieties were first studied by \cite{AB16}, where the 
authors showed that the combinatorics of the matroid associated to the 
arrangement determined much of the geometry of the variety.
The intersection cohomology of matroid Schubert varieties was used in \cite{HW17} to prove Dowling and Wilson's Top
Heavy conjecture for matroids in the realizable case.

The affine reciprocal plane of an essential hyperplane 
arrangement is the intersection of the matroid Schubert variety with the 
affine chart 
$(\C^\times \cup \{\infty\})^n \subseteq (\P^1)^n$. Reciprocal planes 
are studied in \cite{T02,SP06,EPW16,KV19}, were the authors also observe a 
two way street between the combinatorics of arrangements and the geometry of 
their reciprocal planes.
The intersection cohomology of the projectivized reciprocal plane was used in \cite{EPW16} to prove
that the coefficients of the Kazhdan-Lusztig polynomial of a
realizable matroid are non-negative.

In this paper we study the geometry of matroid Schubert varieties 
through the lens of equivariant
compactifications. If $G$ is an algebraic group, then an equivariant compactification of $G$ is a proper variety $X$ containing $G$ as a dense open set, and an action $G \times X \to X$ extending the group law $G \times G \to G$. With the word ``proper'' omitted, we call $X$ an equivariant partial compactification. For example, a toric variety is by definition an equivariant partial compactification of the algebraic torus $T = (\C^\times)^d$. One of the main theorems in toric geometry states that all normal toric varieties arise from polyhedral fans.
A matroid Schubert variety is
an equivariant compactification of the additive group $V = \C^d$, which 
we will call a \emph{vector group}.

To see the equivariant structure, note that $\C \subseteq \P^1$ is an 
equivariant compactification of the additive group $\C$, and so $\C^n 
\subseteq (\P^1)^n$ is an equivariant compactification of $\C^n$. 
Therefore the closure of any subgroup $V \subseteq \C^n$ in $(\P^1)^n$ is an 
equivariant compactification of $V$, because the action of $V$ on 
itself preserves its closure.

The main purpose of this paper is to 
prove the following characterization of which equivariant 
compactifications of vector groups arise as matroid Schubert varieties.

\begin{thm}\label{quick-main-thm}
  An equivariant compactification $Y$ of the vector group $V = \C^d$
  is isomorphic to a matroid Schubert variety if
  and only if $Y$ is normal as a variety, $Y$ has only finitely many 
  orbits, and each orbit contains a point which can be reached by a
  limit $\lim_{t \to \infty} tv$, for $v \in V$.
\end{thm}

The limit condition in the above theorem is analogous to the fact
that any orbit in a normal toric variety can be reached by a
one-parameter subgroup of the torus.
Because the fan corresponding to a normal toric variety is constructed
by considering the limits of one-parameter subgroups, it is natural to look for an analogous correspondence only for equivariant
compactifications of $V$ where every orbit is reached by a
one-variable limit.

In the course of proving the above theorem, we give another 
characterization in which the limit condition is replaced by the 
stronger condition that each orbit admits a normal slice satisfying 
certain properties. The second characterization resembles a key 
geometric property of matroid Schubert varieties: for 
each flat in a hyperplane arrangement, the matroid Schubert variety of the 
restriction is a normally nonsingular slice through the corresponding 
orbit \cite{BHM+20b}.

Finally, we prove an equivalence of categories which generalizes both characterizations to include partial compactifications as well as morphisms between them. The objects in the first category are equivariant partial compactifications of $V$ satisfying the conditions of the above theorem, or the stronger formulation involving normal slices. The objects in the second category we call \emph{partial hyperplane arrangements}, which include all essential hyperplane arrangements as examples.

\subsection{Equivariant compactifications}

We assume all varieties are irreducible and separated
over $\C$. Suppose that $G$ is a commutative linear algebraic 
group, acting on a variety $X$. Given a point $x \in X$, we write $G
\cdot x \subseteq X$ for the orbit and $G_x \subseteq G$ for the
stabilizer.

The main tool we will use throughout the paper is the following notion of a \emph{slice} (\cref{slice-dfn}), which is standard for 
actions of Lie groups and algebraic groups \cite{G50,M57,MY57,P61}. A
(Zariski) slice through $x \in X$ is a $G_x$-stable subvariety
$Z_x \subseteq X$ containing $x$, such that $G \cdot Z_x \subseteq X$ is an open set, and
\[
G \cdot Z_x \cong G \times Z_x/\sim,\quad \text{where $(gh,z) \sim (g,hz)$ for all $g \in G, h \in G_x, z \in Z_x$.}
\]
Geometrically, $Z_x$ is a normal slice through the orbit $G \cdot x$, 
and $G \cdot Z_x$ is neighborhood of $G \cdot x$ that admits a
product structure, similar to a tubular neighborhood. We use the words ``Zariski slice'' to emphasize the difference between the above notion and that of an \'{e}tale slice.

In order to state our results, we make the following abbreviations. Suppose now that $X$ is an equivariant partial compactification of $G$. We say $X$ satisfies
  \begin{itemize}[leftmargin=*]
  \item \FO{} (\textbf{F}inite \textbf{o}rbits) if $X$ has finitely many 
    $G$-orbits,
  \item \OP{} (\textbf{O}ne-\textbf{p}arameter subgroups) if for every 
    $G$-orbit $G \cdot x \subseteq X$, there is a one dimensional 
      algebraic subgroup of $G$ whose closure in $X$ intersects $G\cdot 
      x$, and
  \item \SL{} (\textbf{Sl}ices) if there exists a Zariski slice through every point of $X$.
\end{itemize}

The following is our main result on matroid Schubert varieties, which implies \cref{quick-main-thm}.
\begin{letteredthm}\label{main-thm-compact}
  Suppose that $Y$ is an equivariant compactification of the vector group $V = \C^d$. Then the following are equivalent.
  \begin{enumerate}[leftmargin=*]
    \item \label{main-thm-schubert-variety} $Y$ is equivariantly 
      isomorphic to a matroid Schubert variety.
    \item \label{main-thm-slices} $Y$ is normal and satisfies \FO{} and \SL{}.
    \item \label{main-thm-1psg} $Y$ is normal and satisfies \FO{} and \OP{}.
  \end{enumerate}
\end{letteredthm}

The original aim of this project was to prove that the third statement 
implies the first, however we have found that it is most natural to 
prove that the third statement implies second, and then prove that the 
second implies the first. For this reason, we view the existence of 
slices as a more fundamental property of matroid Schubert varieties.

The study of equivariant compactifications of vector groups was
initiated by \cite{HT99}, and we recommend \cite{AZ20} for a
survey. We will see in \cref{toric-analogy-section} that matroid 
Schubert varieties show 
many parallels to toric varieties, however the study of general equivariant compactifications 
of vector groups has little in common with toric geometry \cite{AZ20}.  
In particular, toric varieties satisfy \FO{}, \OP{}, and \SL{}, whereas
these conditions need not hold for equivariant compactifications of
vector groups, as the following examples show.

  \begin{zb}
    Consider the action of a vector group $V$ of dimension at
      least two on its projective closure $\P (V \oplus \C)$, where the
      action on
      the boundary is trivial. This compactification satisfies \SL{} and \OP{}, but not \FO{}.
\end{zb}
  \begin{zb}[\cite{HT99}]
    Consider the action of the two-dimensional
      vector group $\C^2$ on $\P^2$ where $(a_1,a_2)$ acts
      as
      \[\text{exp}\left(\begin{bmatrix}
        0 & a_1 & a_2\\
        0 & 0 & a_1\\
        0 & 0 & 0\\

      \end{bmatrix}\right) = 
      \begin{bmatrix}
        1 & a_1 & a_2 + \frac{1}{2} a_1^2\\
        0 & 1   & a_1\\
        0 & 0   & 1
      \end{bmatrix}.
      \]
      This action has one two-dimensional orbit (with which we can
      identify $\C^2$), one one-dimensional orbit, and one
      zero-dimensional orbit, so \FO{} holds. However \SL{} and \OP{}
      fail for the one dimensional orbit.
  \end{zb}

\subsection{Equivariant partial compactifications}
We now describe how \cref{main-thm-compact} extends to equivariant 
partial compactifications of vector groups, as well as morphisms between them. We define a \emph{morphism of equivariant compactifications of vector groups} to be a map of varieties, which restricts to a linear map from the first vector group to the second.

Let us first review some hyperplane arrangement terminology. We will
work only with arrangements of linear hyperplanes in a finite
dimensional complex vector
space. We do not consider arrangements of affine hyperplanes. We say
that a hyperplane arrangement is \emph{essential} if the common
intersection of the hyperplanes is zero. A \emph{flat} of a hyperplane arrangement is a linear
subspace of the ambient vector space which can be written as the intersection of several
hyperplanes. We consider the ambient vector space to be a flat,
because it arises from the empty intersection of hyperplanes.

\begin{rmk}
Following the standard convention, we equip the collection of flats
 with the partial order given by reverse inclusion, writing $F
\leq G$ if $F$ and $G$ are flats such that $G\subseteq F$. When the arrangement is essential, this
partial order gives the collection of flats the structure of a finite
geometric lattice, or equivalently, a simple matroid. For our
purposes, the partial order will only be used in \cref{PHA-in-HA} and
in the fact that will refer to the flats of an essential hyperplane
arrangement as the ``lattice of flats.''
\end{rmk}

Viewing an essential hyperplane arrangement as its lattice of flats, we make the following generalization:

\begin{dfn} \label{def:partial-hyperplane-arrangement}
  A \textbf{partial hyperplane arrangement} in $V = \C^d$ is a finite 
  collection $\LL$ of vector subspaces of $V$, such that
  \begin{enumerate}[leftmargin=*]
    \item $\{0\} \in \LL$,
    \item if $F,F' \in \LL$ then $F \cap F' \in \LL$,
    \item for each $F \in \LL$, the set $\{F' \in \LL : F' \subseteq F\}$ 
      is the lattice of flats of an essential hyperplane arrangement in 
      the vector space $F$.
  \end{enumerate}
\end{dfn}

\begin{zb}\label{PHA-in-HA}
  Suppose that $\LL$ is the lattice of flats of an essential
  hyperplane arrangement, and $\mathscr{P} \subseteq \LL$ is an order filter (i.e. an upward closed set under the partial order of reverse inclusion.) Then $\mathscr{P}$ is a partial hyperplane arrangement.
\end{zb}

\begin{zb}
Here we give an example of a partial hyperplane arrangement in
$\mathbb{C}^4$. Let $\LL$ consist of the zero subspace of $\C^4$,
together with the affine cones of the points, lines, and
planes in $\P^3$ listed in \cref{eg:partial-hyperplane-arrangement}.
\end{zb}

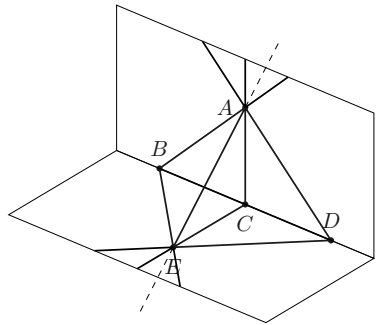
\begin{figure}[h!]
\begin{minipage}[t]{0.4\textwidth}
\strut\vspace*{-\baselineskip}\newline
\resizebox{\textwidth}{!}{
\begin{tabular}{ |c|c|c|c| } 
\hline
Points & Lines & Planes \\
\hline
  $A = [0:0:1:1]$&$AB$&$ABCD$\\
  $B = [0:1:0:1]$&$AC$&$BCDE$\\
  $C = [0:0:0:1]$&$AD$&\\
  $D = [0:-1:0:1]$&$BE$&\\
  $E = [1:0:0:1]$&$CE$&\\
&$DE$&\\
&$AE$&\\
&$BCD$&\\
\hline
\end{tabular}
}
\end{minipage}\hspace{0.1\textwidth}
\begin{minipage}[t]{0.4\textwidth}
\strut\vspace*{-\baselineskip}\newline
\tdplotsetmaincoords{60}{130}
\resizebox{\textwidth}{!}{
\begin{tikzpicture}[scale=1.8,tdplot_main_coords, bullet/.style={circle,inner
sep=1pt,fill=black,fill opacity=1}]
\coordinate (O) at (0,0,0);
\tdplotsetcoord{P}{.8}{55}{60}
\draw[thick, nodes={opacity=1},-] (0,0,0) node[bullet, label=below:{$C$}]{} -- (0,-1,0) node[bullet, label=above:{$B$}]{};
\draw[thick,nodes={opacity=1},-] (0,0,0) -- (1,0,0) node[bullet, 
  label=below:{$E$}]{}
 (0,0,0) -- (0,1,0) node[bullet, label=above:{$D$}]{}
 (0,0,0) -- (0,0,1) node[bullet, label=left:{$A$}]{}
 (0,0,1) -- (1,0,0)
 (1,0,0) -- (0,-1,0)
 (0,0,1) -- (0,-1,0)
 (1,0,0) -- (0,1,0)
 (0,0,1) -- (0,1,0);

\draw[fill=white,fill opacity=0.1,-] 
(0,-1.5,0) -- (0,-1.5,1.5) -- (0,1.5,1.5) -- (0,1.5,0) -- cycle;

\draw[fill=white,fill opacity=0.1,-] 
(0,-1.5,0) -- (1.5,-1.5,0) -- (1.5,1.5,0) -- (0,1.5,0) -- cycle;

\draw[dashed,nodes={opacity=1},-] (0,0,1) -- (-0.5, 0, 1.5);
\draw[thick,nodes={opacity=1},-] (0,0,1) -- (0,-0.5, 1.5)
 (0,0,1) -- (0,0, 1.5)
 (0,0,1) -- (0,0.5, 1.5);

\draw[dashed,nodes={opacity=1},-] (1,0,0) -- (1.5, 0, -0.5);
\draw[thick,nodes={opacity=1},-] (1,0,0) -- (1.5,-0.5, 0)
(1,0,0) -- (1.5,0, 0)
(1,0,0) -- (1.5,0.5, 0);

\end{tikzpicture}
}
\end{minipage}
\caption{The projectivization of a partial hyperplane arrangement in $\C^4$.}\label{eg:partial-hyperplane-arrangement}
\end{figure}

\begin{zb}
  Here we give an example of a partial hyperplane arrangement which
  cannot be realized as an order filter in the lattice of flats of a
  hyperplane arrangement. Consider the partial hyperplane
  arrangement $\LL$ in $\C^3$ consisting of the proper coordinate
  subspaces, and a general line. Any hyperplane passing through the
  general line will intersect one of the coordinate hyperplanes in a
  non-coordinate line. Therefore if there is hyperplane
  arrangement whose lattice of flats contains $\LL$, then $\LL$
  cannot be upward closed.
\end{zb}

\begin{dfn}\label{morphism-of-arrangements}
  Suppose that $\LL_i$ is a partial hyperplane arrangement in a vector
  group $V_i$ for $i=1,2$, and $T:V_1 \to V_2$ is a linear map. Then we say $T$ is a \textbf{morphism of partial hyperplane arrangements} if
  \begin{enumerate}[leftmargin=*]
    \item \label{flats-to-flats} for each $F_1 \in \LL_1$ there exists $F_2 \in \LL_2$ such that $T(F_1) \subseteq F_2$,
    \item for each $F_1 \in \LL_1$ and $F_2 \in \LL_2$, $T^{-1}(F_2) \cap F_1 \in \LL_1$.
  \end{enumerate}
\end{dfn}

\begin{zb}\label{morphisms-of-hyperplane-arrangements}
  In the case where $\LL_i$ is the lattice of flats of a hyperplane
  arrangement $\A_i$, then $T$ is a morphism of
  partial hyperplane arrangements if and only if the preimage of each
  hyperplane in $\A_2$ is either a hyperplane in $\A_1$ or is $V_1$.
\end{zb}

The following is our main result, in maximal generality.
\begin{letteredthm}\label{main-thm-noncompact}
  There is a fully faithful embedding of categories from partial 
  hyperplane arrangements to equivariant partial compactifications of 
  vector groups, such that the following are equivalent for an 
  equivariant partial compactification $Y$.
  \begin{enumerate}[leftmargin=*]
    \item $Y$ arises from a partial hyperplane arrangement.
    \item \label{main-thm-noncompact-slices} $Y$ is normal and satisfies \FO{} and \SL{}.
    \item \label{main-thm-noncompact-1psg} $Y$ is normal and satisfies \FO{} and \OP{}.
    \end{enumerate}
\end{letteredthm}

Given an equivariant partial compactification $Y$ of $V$ arising from a 
partial hyperplane arrangement $\LL$, then $\LL$ can be recovered as the 
collection of all stabilizers of points in $Y$. See 
\cref{prop:stabs-are-flats}.

\subsection{Analogy with toric varieties}\label{toric-analogy-section}

For the remainder of the introduction, we say that an equivariant
partial compactification of $V$ satisfying any of the equivalent conditions of \cref{main-thm-noncompact} is a \emph{linear $V$-variety}.
In this section, we explain the analogy between \cref{main-thm-noncompact} and the correspondence between normal toric varieties and polyhedral fans. From now on, we assume all toric varieties are normal.

Toric varieties satisfies \FO{}, \SL{}, and \OP{}, and once we impose
these conditions onto an equivariant partial
compactification of a vector group, there is dictionary
(\cref{toric-analogy-table})
which is similar to the dictionary between toric varieties and
fans. In both cases the idea is to cover the variety with
``simple'' open sets. The main difference is that these
open sets are affine in the torus case and non-affine in the vector
group case.

\begin{figure}
\begin{tabular}{L{0.05\linewidth}|L{0.19\linewidth} L{0.19\linewidth}|L{0.19\linewidth} L{0.19\linewidth}} 

 & \multicolumn{2}{c}{Toric varieties} & \multicolumn{2}{c}{Linear $V$-varieties}\\
\hline
Sec. & Combinatorics & Geometry & Combinatorics & Geometry \\
\hline
\ref{intro-analogy-1psg} & $N$ & $T$ & $V$ & $V$ \\

\hline
\ref{intro-analogy-full-dim-cones} & Full dimensional cones in $N_\R$ & 
Affine toric varieties with no torus factors & Essential hyperplane 
arrangements in $V$ & Matroid Schubert varieties \\

\hline
\ref{intro-analogy-cones} & Cones in $N_\R$ & Affine toric varieties &
Essential hyperplane arrangements in $F \subseteq V$ &
Simple linear $V$-varieties \\

\hline
\ref{intro-analogy-fans} & Fans in $N_\R$ & Toric varieties & Partial hyperplane arrangements in $V$ & Linear $V$-varieties \\

\hline
\ref{intro-analogy-orbit-cone} & Cones in a fan $\Sigma$ & Orbits, distinguished points, and invariant affine opens in $X_\Sigma$ & Flats in a partial hyperplane arrangement $\LL$ & Orbits, distinguished points, and invariant simple opens in
$Y_\LL$ \\
\hline
\end{tabular}
\caption{Correspondences for toric varieties versus correspondences for 
  linear $V$-varieties.}
\label{toric-analogy-table}
\end{figure}

\subsubsection{One-parameter subgroups}\label{intro-analogy-1psg} Let $T$ be an algebraic torus. A one-parameter subgroup of $T$ is an algebraic group homomorphism from $\C^\times$ to $T$. The one-parameter subgroups of $T$ form a finitely generated free abelian group $N$, and write $N_\R = N \otimes_\Z \R$ for the corresponding real vector space. Let $V$ be a vector group. By a one-parameter subgroup of $V$, we mean an algebraic group homomorphism from $\C$ to $V$. The one-parameter subgroups of $V$ naturally correspond to the elements of $V$, so $V$ will play the role of both $T$ and $N$.

\subsubsection{Full dimensional cones and essential hyperplane 
arrangements}\label{intro-analogy-full-dim-cones}

The toric varieties arising from full dimensional cones are 
exactly the affine toric varieties that have no torus factors.
If $\sigma \subseteq N_\R$ is full dimensional (strictly convex rational polyhedral) cone, then there is a 
canonical embedding of tori $T \subseteq 
\prod_{u \in \mathscr{H}} \C^\times$ where $\mathscr{H}$ is the unique minimal basis of the dual semigroup. Note that this embedding is only canonical when $\sigma$ is full dimensional. The 
corresponding toric variety is the closure of $T$ in 
$\prod_{u \in \mathscr{H}}(\C^\times \cup \{0\})$.
If $\A$ is an essential hyperplane arrangement in $V$, then there is a 
canonical embedding of vector groups $V \subseteq \prod_{H \in \A} V/H$. The 
corresponding matroid Schubert variety is the closure of $V$ in $\prod_{H \in 
\A} (V/H \cup \{\infty\})$.

\subsubsection{Simple partial 
compactifications}\label{intro-analogy-cones} Suppose an algebraic group $G$ acts on a 
variety $X$ with finitely many orbits. We say that $X$ is \emph{simple} if 
there is a unique closed orbit. Since the orbits form a finite stratification, $X$ can be covered with 
simple $G$-stable open sets.
Simple toric varieties are exactly affine toric varieties by
\cite[Corrollary 2]{S74}, and every affine toric variety arises from a sublattice $N' \subseteq N$ and a full dimensional cone $\sigma \subseteq N' \otimes_\Z \R$.
Simple linear $V$-varieties are
not affine, however by \cref{prop:proper-slice} and \cref{main-thm-compact}, every simple linear $V$-variety arises from a vector subspace $F \subseteq V$ and
an essential hyperplane arrangement in $F$. 

\subsubsection{Partial compactifications}\label{intro-analogy-fans}
Toric varieties are constructed from affine toric varieties by
gluing according to the fan, and likewise linear $V$-varieties are
constructed from simple linear $V$-varieties by gluing according to
the partial hyperplane arrangement. See \cref{construction-comparison}
for more details.

\subsubsection{Orbit correspondences}\label{intro-analogy-orbit-cone}
Let $\Sigma$ be a fan in $N_\R$ corresponding to a toric variety 
$X_\Sigma$. Let $\sigma^\circ$ denote the relative interior of a cone 
$\sigma \in \Sigma$. That is, $\sigma^\circ = \sigma \setminus 
\bigcup \tau$ where the union is over $\tau \in \Sigma$ such that $\tau 
\subsetneq \sigma$. The cones $\sigma \in \Sigma$ are in one-to-one 
correspondence with distinguished points $x_\sigma \in X_\Sigma$, given 
by 
\[
N \cap \sigma^\circ = \{\lambda \in N : \lim_{t \to 0}\lambda(t) = x_\sigma\}.
\]
Let $\LL$ be a partial hyperplane arrangement corresponding to a linear 
$V$-variety $Y_\LL$. Given $F \in \LL$, write $F^\circ = F \setminus 
\bigcup F'$ where the union is over $F' \in \LL$ such that $F' \subsetneq 
F$. The flats $F \in \LL$ are in one-to-one correspondence with 
distinguished points $y_F \in Y_\LL$, given by
\[
F^\circ = \{v \in V : \lim_{t \to \infty}tv = y_F\}.
\]
In both cases, each orbit contains exactly one distinguished point. 
Therefore cones (resp. flats) also correspond to orbits, and to simple 
invariant open sets in $X_\Sigma$ (resp. $Y_\LL$). See
\cref{sec:orbit-flat-correspondence} for more details.

\subsection{Structure of the paper}
In \cref{sec:slice-thm} we prove some consequences of \FO{} and \SL{}
which will be used throughout the paper, and we prove the equivilance of\cref{main-thm-noncompact-slices} and \cref{main-thm-noncompact-1psg} in \cref{main-thm-noncompact}. In
\cref{sec:orbit-flat-correspondence} we prove the analog of
the orbit-cone correspondence for linear
$V$-varieties. In section 
\cref{sec:structure-of-linear-compactifications} we prove
\cref{main-thm-compact}, and in
\cref{sec:structure-of-linear-V-varieties} we prove 
\cref{main-thm-noncompact}. The appendix describes matroid Schubert 
varieties in coordinates.

\subsection{Acknowledgments}
The author would like to thank his advisor Botong Wang for his advice and encouragement throughout this project, as well as Connor Simpson for reading a previous draft and providing useful feedback.

\section{Slices} \label{sec:slice-thm}

\subsection{Slices of group actions} \label{sec:slice-def}
We begin by reviewing the definition of homogeneous fiber spaces, following \cite[Chapter II.4.8]{S94}.
Suppose that $G$ is an algebraic group, $H$ is an algebraic subgroup,
and $Z$ is a quasiprojective variety on which $H$ acts. Then there
exists a variety $G *_H Z$ called the \emph{homogeneous fiber space},
which parameterizes equivalence classes in $G \times Z$, where
\[
(gh,z) \sim (g,hz),\quad \text{for all $g \in G, h \in H, z \in Z$}.
\]
There is a canonical map $G \times Z \to G *_H Z$ sending a point $(g,z)$ to its equivalence class, which we write as $[g,z]$.
The universal property which characterizes $G *_H Z$ is the following: if
$
\pi:G \times Z \to X
$
is a map of varieties such that $\pi(gh, z) = \pi(g, hz)$ for all $g \in G, h \in H, z \in Z$, then there is a unique factorization as follows.
\begin{center}
\begin{tikzcd}
G \times Z \arrow[rr, "\pi"] \arrow[rd] &                    & X \\
                                        & G *_H Z \arrow[ru] &  
\end{tikzcd}
\end{center}
From the universal property, we see that there is a canonical map
\[
\tau:G *_H Z \to G/H,\quad [g,z] \mapsto gH,
\]
with each fiber isomorphic to $Z$. We call $\tau$ the 
\emph{canonical fibration}. If $H \subseteq G$ is
normal and has a splitting $s:G \to H$, then there is a $G$-equivariant
isomorphism
\[
G *_H Z \cong G/H \times Z,\quad [g,z] \mapsto (gH, s(g) \cdot z),
\]
which makes the following diagram commute, where $\text{pr}_1$ is the projection.

\begin{center}
\begin{tikzcd}
G *_H Z \arrow[r, "\cong"] \arrow[rd, "\tau"] & G/H \times Z \arrow[d, "\text{pr}_1"] \\
                                              & G/H                                  
\end{tikzcd}
\end{center}

\begin{rmk}\label{splitting-product}
  For the remainder of the paper we will take $G$ to be
  commutative. Therefore it is possible for us to avoid defining
  $G *_{H} Z$ by choosing splittings and working with $G/H
  \times Z$ instead. While $G/H \times Z$ is a simpler
  construction, we have found that thinking in terms of the more
  canonical construction $G *_H Z$ shortens and clarifies the
  rest of the paper enough to make it worthwhile.
\end{rmk}

We need the following lemmas, which are formal consequence of the definitions.

\begin{lem}[Associativity of $*$]\label{associativity}
  Suppose that $H' \subseteq H \subseteq G$ are closed subgroups and
  $H'$ acts on a variety $Z'$. Then there is a natural isomorphism
  \[
    G *_{H} (H *_{H'} Z') \cong G *_{H'} Z',\quad [g,[h,z]] \mapsto 
    [gh,z].
  \]
\end{lem}
\begin{lem}[Orbits and stabilizers of $G *_H Z$] 
  \label{orbits-and-stabs-of-GHZ} \leavevmode
  \begin{enumerate}[leftmargin=*]
    \item There is a one-to-one correspondence between $G$-orbits in $G 
      *_{H} Z$ and $H$-orbits in $Z$ which sends $G \cdot [v,z]$ to $H \cdot z.$
    \item Suppose that $G$ is commutative, and $x = [v,z] \in G *_H Z$. Then 
      the stabilizers $G_x$ and $H_z$ coincide as subgroups of $G$.
  \end{enumerate}
\end{lem}

\begin{dfn}\label{slice-dfn}
Suppose that $X$ is an algebraic variety with a $G$-action. If $x \in
X$ is a point with stabilizer $G_x$, we say that a $G_x$-stable
locally closed subvariety $Z_x \subseteq X$ containing $x$ is a \textbf{(Zariski) slice} at $x$ if the natural map
  \[
  G*_{G_x}Z_x \to X,\quad [g,z] \mapsto g \cdot z
  \]
  is a $G$-equivariant Zariski open embedding.
\end{dfn}

The point $x$ is in the image of $G *_{G_x} Z_x$, so we have
that $G *_{G_x} Z_x$ is identified with a $G$-stable neighborhood of
the orbit $G \cdot x$.

We will often use the following criterion for open embeddings to prove the 
existence of slices.

\begin{thm}[Zariski's main theorem]\label{thm:open-criterion}
  Suppose that $\pi:X \to Y$ is a morphism of varieties which is
  birational and injective on closed point, and that $Y$ is
  normal. Then $\pi$ is an open embedding.
\end{thm}
For the above formulation, we refer to \cite[Exercise 29.6.D]{V17} and the surrounding discussion.
For our purposes, checking injectivity on closed points can be rephrased as follows.

\begin{lem}\label{injectivity-criterion}
  Suppose that $x \in X$ and $Z_x \subseteq X$ is a $G_x$-stable subvariety 
  containing $x$. Then $G *_{G_x} Z_x \to X$ is injective on closed 
  points if
  $g_1 \cdot z_1 = g_2 \cdot z_2$ implies $g_2^{-1}g_1 \in G_x$
  for all $g_1,g_2 \in G$ and $z_1,z_2 \in Z_x$.
\end{lem}

\subsection{Partial compactifications with slices}\label{compactifications-with-slices}
For this subsection let $G$ be a commutative linear algebraic group,
and $X$ a normal equivariant partial compactification of $G$ such that
\FO{} and \SL{} hold. We will first collect some basic consequences.

\begin{lem}\label{slice-coset}
  If $x \in X$ has a slice $Z_x$, then $Z_x \cap G$ is a coset of $G_x$.
\end{lem}
\begin{proof}
  We have that $G$ is contained in any
  invariant open neighborhood of $X$, and the natural map
  \[
  G *_{G_x} Z_x \to X
  \]
  is a $G$-equivariant open embedding, so there must be $[v,z] \in G
  *_{G_x} Z_x$ mapping to $G$. Therefore $Z_x \cap G \not=
  \emptyset$. We also have that $G *_{G_x} Z_x \cong G/G_x \times Z_x$
  embeds as a Zariski open set in the variety $X$, so $Z_x$ and thus
  $Z_x \cap G$ is
  irreducible of dimension $\dim G - \dim G/G_x = \dim G_x$. Finally we
  note that the only irreducible $G_x$-invariant closed subsets of $G$ of
  dimension $\dim G_x$ are cosets.
\end{proof}

There is a special point in each orbit corresponding to the trivial coset:

\begin{dfn}\label{distinguished-point}
We say that a point $x \in X$ is \textbf{distinguished} if it has a
slice containing the identity of $G$. 
\end{dfn}

It follows from \cref{slice-coset} that every orbit contains a 
distinguished point, and we will see in \cref{closure-of-stabilizer} 
that every orbit contains at most one distinguished point.

The orbits of $X$ form a finite stratification, so each orbit
    $G\cdot x$ has a unique smallest $G$-invariant open neighborhood
defined as follows.
\begin{dfn}\label{U-x-def}
The \textbf{minimal $G$-invariant neighborhood} $U_x$ of $x \in X$ is given
by the union of all orbits $G\cdot y$ such that $G\cdot x \subset
\overline{G\cdot y}$.
\end{dfn}

There exists a unique slice through $x \in X$ contained in $U_x$, defined as follows:
\begin{dfn}
  The \textbf{minimal slice} through $x \in X$ is $Z_x \cap U_x$, where $Z_x$ is any slice of $x$.
\end{dfn}
It follows by \cite[Proposition II.4.21]{S94} that the minimal slice through $x$ is indeed a slice,
and uniqueness of the minimal slice for distinguished points (the case
of an arbitrary point follows easily) comes from the following observation:

\begin{lem}\label{closure-of-stabilizer}
  The minimal slice $Z_x$ at a distinguished point $x \in X$ is the
  closure of $G_x$ in $U_x$.
\end{lem}

\begin{proof}
  Because $x$ is distinguished, $Z_x$ contains the identity of $G$. 
  Since $G \cap Z_x$ is a coset of $G_x$ by \cref{slice-coset}, $G \cap 
  Z_x = G_x$. Thus the closure of $G_x$ in $U_x$ is contained
  in $Z_x$. Because $G *_{G_x} Z_x \cong G/G_x \times Z_z$ embeds as
  an open set in $X$, we get that $Z_x$ is irreducible of dimension $\dim
  G - \dim G/G_x = \dim G_x$, and closed in
  $U_x$. Therefore
   the closure of $G_x$ in $U_x$ equals $Z_x$, since they are both irreducible closed 
   subvarieties of $U_x$ of the same dimension.
\end{proof}

\begin{zb}[Distinguished points and minimal slices of $\P^1$]\label{P1-eg}\leavevmode
  \begin{enumerate}[leftmargin=*]
    \item Consider $\P^1 = \C^\times \cup \{0\} \cup \{\infty\}$ as an 
  equivariant compactification of $\C^\times$. The distinguished points 
  are $1,0,$ and $\infty$ with minimal invariant open neighborhoods 
  $\C^\times,\C^\times \cup \{0\},$ and $\C^\times \cup \{\infty\}$ 
  respectively. The minimal slices are $\{1\}, \C^\times \cup \{0\},$ 
  and $\C^\times \cup \{\infty\}$ respectively. Note that $\P^1$ is a 
  non minimal slice through both $0$ and $\infty$.
\item Consider $\P^1 = \C \cup \{\infty\}$ as an equivariant 
  compactification of $\C$. The distinguished points are $0$ and 
      $\infty$ with minimal invariant open neighborhoods $\C$ and 
      $\P^1$ respectively. The minimal slices are $\{0\}$ and $\P^1$ 
      respectively. \qedhere
  \end{enumerate}
\end{zb}
In \cref{appendix} we demonstrate the notions developed in this section 
for matroid Schubert varieties.

We now prove that the class of varieties we are working with is
closed under taking slices:
\begin{lem}\label{closed-under-slices}
  If $x \in X$ is a
  distinguished point, then the minimal slice $Z_x$ is a normal partial
  compactification of $G_x$ satisfying \FO{} and \SL{}.
\end{lem}
\begin{proof}
  By \cref{closure-of-stabilizer}, $Z_x$ is an equivariant partial compactification of 
  $G_x$, and $Z_x$ is normal by \cite[Proposition II.4.22]{S94}.
  By \cref{orbits-and-stabs-of-GHZ}, the $G_x$-orbits of $Z_x$ correspond to the $G$-orbits of an open set in $X$, so $Z_x$ has finitely many $G_x$-orbits.
  Finally we check that $Z_x$ has slices. For a point $y \in Z_x$,
  simply take the slice $Z_y$ through $y$ in $X$. Since $y \in Z_x \subseteq 
  U_x$, $G \cdot x \subseteq \overline{G \cdot y}$ by 
  definition of $U_x$. Therefore $G_y \subseteq G_x$, so $Z_y \subseteq 
  Z_x$ by \cref{closure-of-stabilizer}. Now consider the diagram

  \begin{center}
    \begin{tikzcd}
G*_{G_y}Z_y \arrow[r]               & X             \\
G_x *_{G_y} Z_y \arrow[r] \arrow[u] & Z_x \arrow[u]
\end{tikzcd}
    \end{center}
  where the horizontal maps are the open embeddings $[v,z] \mapsto v 
  \cdot z$, and the
  left vertical map is given by $[v,z] \mapsto [v,z]$. We wish to show
  that the bottom arrow is an open embedding, for which we use
  \cref{thm:open-criterion}. The bottom arrow restricts to 
  an isomorphism
  $
  G_x *_{G_y} G_y \cong G_x,
  $
  so it is birational. All three arrows except the bottom one are already known to be
  injective on closed points, so the bottom arrow is injective on closed 
  points.
\end{proof}

\begin{rmk}
  It follows from \cite[Proposition II.4.21]{S94} that every orbit closure satisfies \SL{} for the group $G/G_x$. So modulo normality, the class of varieties studied in this section is also closed under taking orbit closures.
\end{rmk}

\subsection{Topology of orbit stratification}
In the previous section we studied partial compactifications of 
tori and vector groups simultaneously. In this section, we will use 
properties of vector groups which fail for 
tori.

\begin{prop}\label{prop:proper-slice}
  If $X$ is an equivariant partial compactification of a vector group $V$ satisfying \FO{} and \SL{}, and $x \in X$ is a distinguished point, then the
  minimal slice $Z_x$ is proper and has $x$ as the unique $V_x$-fixed point.
\end{prop}

The proof follows from a general topological observation about
varieties stratified into affine spaces.
We say that an \emph{algebraic cell decomposition} of a variety $X$
is a partition $X = \sqcup_\alpha S_\alpha$ into finitely many locally
closed
subvarieties $S_\alpha$ called \emph{cells}, such that each
cell is isomorphic to an affine space and the closure of a cell
is a union of cells.

\begin{lem}\label{lem:Z-proper}
  Suppose that $Z$ is a connected variety with an algebraic cell decomposition
  that has at least one zero dimensional cell. Then $Z$ is proper and
  has exactly one zero dimensional cell.
\end{lem}
\begin{proof}
  Consider the singular cohomology with compact support $H^i_c(Z; \Q)$.
  Since $Z$ is connected, we have that $H^0_c(Z; \Q)$ is zero if $Z$
  is not proper and one dimensional if $Z$ is proper.
  The lemma follows from the well known fact that
  \[
    \dim H^{2i}_c(Z; \Q) = \#\{\text{$i$-dimensional cells in $Z$}\}.
  \]
  One way to prove the above equation is by inducting on the
  number of cells as follows. Suppose $S$ is a cell of lowest
  dimension in $Z$ and $U$ is its open complement. Since $S \cong
  \mathbb{A}^r$ for some $r$, $H^{2r}_c(S; \Q) = \Q$ and $H^{i}_c(S;
  \Q) = 0$ for $i \not= 2r$.  Then the above equation follows
  from induction using long exact sequence
  \[
  \ldots \to H^i_c(U; \Q) \to H^i_c(Z; \Q) \to H^i_c(S; \Q) \to
  H^{i+1}_c(U; \Q) \to \ldots.\qedhere
  \]
\end{proof}

\begin{proof}[Proof of \cref{prop:proper-slice}]
  We have by \cref{closed-under-slices} that $Z_x$ has finitely many
  $V_x$ orbits. Each orbit of $V_x$ is isomorphic to an
  affine space, and as is true of any algebraic group action, orbits are
  locally closed and the closure of an orbit is a union of
  orbits. Thus the $V_x$-orbits of $Z_x$ form an algebraic cell
  decomposition. Since $x \in Z_x$ is a zero dimensional cell, the
  proposition follows from \cref{lem:Z-proper}.
\end{proof}

\begin{rmk}
  In the notation of \cref{prop:proper-slice}, it follows that the 
  minimal slice through $x$ is the unique slice through $x$, as opposed 
  to the torus case. See \cref{P1-eg}.
\end{rmk}
\begin{rmk}
  In the case where $X$ is a toric variety with torus $T$, the minimal slice $Z_x$ through a point $x \in X$ is not proper but
  rather affine. However $Z_x$ still has $x$ as the
  unique $T_x$-fixed point for a different reason. This is due to the 
  fact that disjoint $T$-invariant 
  closed sets in an affine $T$-variety can be separated by an invariant 
  function \cite[Lemma 6.1]{D03}. Since $Z_x$ has a dense $T_x$-orbit, 
  all invariant functions are constant. Therefore all invariant closed 
  sets intersect, so $x$ is the only $T_x$-fixed point.
\end{rmk}

\subsection{Slices and one-parameter subgroups}\label{slices-and-1psg}

In this section we prove that
\cref{main-thm-noncompact-slices} and \cref{main-thm-noncompact-1psg} in \cref{main-thm-noncompact}
  are equivalent. We break 
the proof into two lemmas.

\begin{dfn}\label{V-circ-def}
Suppose that $X$ is an equivariant partial compactification of a vector 
group $V$, and $x \in X$. Define
\[
    V_x^\circ = V_x \setminus \bigcup V_y
\]
  where the union is over $y \in X$ such that $V_y \subsetneq V_x$.
\end{dfn}

\begin{lem}\label{slices-to-1psg}
  Suppose $X$ is a normal equivariant partial compactification of a 
  vector group $V$, satisfying \FO{} and \SL{}. Let $x \in X$ be a 
  distinguished point, and let $v \in V$. Then $\lim_{t \to \infty} tv = x$ if and only
  if $v \in V_x^\circ$. In particular, $X$ satisfies \OP{}.
\end{lem}
\begin{proof}
  Suppose that $v \in V_x^\circ$. By \cref{prop:proper-slice}, $Z_x$ is 
  proper, so $\lim_{t \to \infty} tv$ must converge to a boundary point of $Z_x 
  \supseteq V_x$. In addition, $v$ lies in the stabilizer of $\lim_{t \to \infty} tv$, so $\lim_{t \to \infty} tv$ be a
  $V_x$-fixed point. By \cref{prop:proper-slice}, $x$ is the unique $V_x$-fixed point in
  $Z_x$, so $\lim_{t \to \infty} tv = x$. To prove the other direction, we note that $V$ is partitioned into sets of the form $V^\circ_y$ for $y \in X$, so if $v \not\in V^\circ_x$ then $\lim_{t \to \infty} tv = y$ for some $y \not= x$.
\end{proof}

\begin{lem}\label{1psg-to-slices}
  Suppose that $X$ is a normal equivariant partial compactification of $V$ satisfying \FO{} and \OP{}. Then $X$ satisfies \SL{}.
\end{lem}

\begin{proof}
  We wish to construct a slice through a point $x \in X$. We first explain
  why it is enough to show that the quotient map
  \[
  V \to V/V_x
  \]
  extends to a $V$-equivariant map
  \[
  \tau:U_x \to V/V_x,
  \]
  where $U_x$ is the minimal invariant open neighborhood of 
  \cref{U-x-def}.
  We can assume without loss of generality that $\tau(x) = 0$, since 
  the translation of a slice is a slice.
  Setting $Z_x := \tau^{-1}(0)$, we have that $V_x$ acts on
  $Z_x$ and $x \in Z_x$. To show that 
  $Z_x$ is a slice, we must check that
  the natural map
  \[
    V *_{V_x} Z_x \to X,\quad [v,z] \mapsto v \cdot z
  \]
  is an open embedding. For this we use \cref{thm:open-criterion}. As
  before we have that $V*_{V_x} Z_x \to X$ restricts to an isomorphism
  $V *_{V_x} V_x \cong V$, so it is birational. By \cref{injectivity-criterion} we must show that if $v_1,v_2 \in V$
  and $z_1,z_2 \in Z_x$ such that
  \[
  v_1\cdot z_2 = v_1 \cdot z_2,
  \]
  then $v_1 - v_2 \in V_x$. We check this by applying $\tau$:
  \begin{align*}
    \tau(v_1) \cdot \tau(z_1) &= \tau(v_2) \cdot \tau(z_2) &&\text{by equivariance of $\tau$,}\\
    \tau(v_1) &= \tau(v_2) &&\text{because $z_1,z_2 \in Z_x = \tau^{-1}(0)$,}\\
    v_1-v_2 &\in V_x &&\text{because $\tau$ extends $V \to V/V_x$.}
  \end{align*}

  Next we show how to construct $\tau$. We wish to construct a $V$-equivariant map
  \[
  \Sym(V/V_x)^\vee \to H^0(U_x, \OO_X),
  \]
  so it is enough to show that if $f \in V^\vee$ vanishes on $V_x$, then
  $f$ can extend to $U_x$. Since $U_x$ is normal, it suffices to show
  that
  $f$ does not have a pole along any codimension one orbit. Let $L
  \subseteq V$ be a one dimensional vector subspace of $V$, and let
  $y$ be the boundary point of $L$ in $X$. Assume that $y \in U_x$. By
  our assumption that $X$ satisfies \OP{}, it is enough to show that $f$ does not have a
  pole along $V \cdot y$. Since the action of $L$ fixes the boundary
  of $L$, $L \subseteq V_y$. We also have $V_y \subseteq V_x$ by
  definition of
  $U_x$. Therefore $f$ vanishes
  on $L$. Now let $L'$ denote the translation of $L$ by a generic
  vector, and $y'$ the boundary point of $L'$. Since $f$ is linear,
  $f$ is constant and nonzero on $L'$. Thus $f^{-1}$ is constant and
  nonzero on $L'$. If $f^{-1}$ is undefined
  at $y'$, then since $y'$ is generic in $V \cdot y$, $f$ has a zero
  along $V \cdot y$. If on the other hand $f^{-1}$ is defined at $y'$,
  then $f^{-1}$ cannot vanish at $y'$ by continuity, so $f$ does not
  have a pole along $V \cdot y$.
\end{proof}

This completes the proof that the statements
\cref{main-thm-noncompact-slices} and \cref{main-thm-noncompact-1psg} in \cref{main-thm-noncompact} are equivalent.

\section{The orbit-flat correspondence}\label{sec:orbit-flat-correspondence}

Now that we have proved the equivalence of the statements
\cref{main-thm-noncompact-slices} and 
\cref{main-thm-noncompact-1psg} in \cref{main-thm-noncompact},
we will refer to an 
equivariant partial compactification of a vector group $V$ which 
satisfies either of these conditions as a \emph{linear 
$V$-variety}.
In \cref{slices-to-1psg}, we showed that
  if $X$ is a linear $V$-variety and $x \in X$ is a distinguished 
  point, then $v \in V_x^\circ$ (\cref{V-circ-def}) if and only
  if $ \lim_{t \to \infty} tv = x$.
As a consequence, we have:
\begin{cor}\label{orbit-flat}
  If $X$ is a linear $V$-variety, then there is a canonical bijection 
  between any two of the following sets.
  \begin{itemize}[leftmargin=*]
  \item Orbits of $X$
  \item Distinguished points of $X$
  \item Stabilizers of points of $X$
\end{itemize}
Moreover, each of the above sets is functorial on the category of
normal equivariant partial compactifications of vector groups, and the
bijections between them are natural.
\end{cor}
\begin{proof}
  The correspondence between orbits and distinguished points is
  automatic, and we have by \cref{slices-to-1psg} that any
  distinguished point $x \in X$ can be recovered from its stabilizer
  $V_x$ by taking the limit $\lim_{t \to \infty} tv$ for $v \in
  V_x^\circ$. Thus all three sets are in correspondence.

  Suppose that $T$ is a morphism of linear vector group varieties. It
  is automatic that orbits are mapped inside of orbits and stabilizers
  are mapped inside of stabilizers. Since distinguished points are the
  set of points that arise as limits of one-parameter subgroups,
  distinguished points are mapped to distinguished points. Naturality
  of these correspondences follows formally.
\end{proof}

\begin{dfn}\label{def:L-of-X}
Let $X$ be a linear $V$-variety. The \textbf{partial hyperplane arrangement $\LL(X)$ associated to $X$} is the collection
of stabilizers of points in $X$.
\end{dfn}

To justify the definition of $\LL(X)$, we will prove:
\begin{prop}\label{prop:stabs-are-flats}\leavevmode
  \begin{propenum}[leftmargin=*]
    \item \label{stabs-are-flats-compact} If $X$ is a proper linear $V$-variety, then $\LL(X)$ is the
  collection of flats of an essential hyperplane arrangement in $V$.

\item \label{stabs-are-flats-noncompact} If $X$ is a linear $V$-variety then $\LL(X)$ is a
  partial hyperplane arrangement in $V$.
  \end{propenum}
\end{prop}

By combining \cref{orbit-flat} and \cref{prop:stabs-are-flats} we have
a natural one-to-one correspondence between the orbits of $X$ and the
flats of its relative hyperplane arrangement $\LL(X)$, as described in \cref{intro-analogy-orbit-cone}.

\begin{lem}\label{closed-under-intersection}
  If $X$ is a linear $V$-variety, then $\LL(X)$ is closed under 
  intersections.
\end{lem}

\begin{proof}
Suppose that $x,y \in X$ are distinguished points, and consider the
action of $V_x \cap V_y$ on the closure $\overline{V_x \cap V_y}$ in
$X$. Let $Z_x$ be the minimal slice through $x$. Then $V_x \subseteq
Z_x$ by \cref{closure-of-stabilizer}, so $\overline{V_x \cap V_y}$ is
a closed subvariety of $Z_x$. Then by \cref{prop:proper-slice}, $Z_x$
is proper, so $\overline{V_x \cap V_y}$ is proper. By the Borel fixed
point theorem \cite[Chapter 21.2]{H75}, there exists a $(V_x \cap
V_y)$-fixed point $z \in \overline{V_x \cap V_y}$.  Thus $V_x \cap V_y
\subseteq V_z$.  To show the opposite inclusion, note that
    \[
    z \in \overline{V_x \cap V_y} \subseteq Z_x \cap Z_y \subseteq U_x 
    \cap U_y,
    \]
    where $U_x$ is the minimal invariant neighborhood.
    Therefore by definition of $U_x$, we have $x,y \in 
    \overline{V \cdot z}$, and thus $V_z \subseteq V_x \cap V_y$.
\end{proof}

To prove \cref{stabs-are-flats-compact}, it 
now suffices to prove that any stabilizer $V_x \subsetneq V$ 
is the intersection of the codimension one stabilizers 
containing it. For this we need the following lemmas.

\begin{lem}\label{extending-functions}
  Suppose that $G$ is a linear algebraic group acting on a variety $X$, 
  and $Z_x$ is a slice through $x \in X$. Then any regular function on 
  $G \cdot x$ extends to the neighborhood $G *_{G_x} Z_x \subseteq X$.
\end{lem}
\begin{proof}
  This follows from the universal property of $G *_{G_x} Z_x$ applied to 
  the map $G \times Z_x \to \C$ given by $(g,z) \mapsto f(g \cdot x)$ 
  where $f$ is a regular function on $G \cdot x$.
\end{proof}
\begin{lem}\label{codim-one-boundary}
  If $U$ is a connected algebraic variety which has nonconstant global
  regular functions, and $U \subseteq K$
  is a compactification, then the boundary $K \setminus U$
  has has an irreducible component of codimension one in $K$.
\end{lem}
\begin{proof}
  Assume for a contradiction that every component of $K \setminus U$ has 
  codimension at least two. Consider the inclusion of the normalizations 
  $\tilde{U} \subseteq \tilde{K}$. Since $U$ has nonconstant global regular
  functions, then so must $\tilde{U}$. The normalization map is finite
  and therefore preserves the codimension of the boundary, so
  $\tilde{K}\setminus \tilde{U}$ is a closed set of codimension at
  least two. Thus any regular function on $\tilde{U}$ extends to
  $\tilde{K}$, so we get that the proper variety
  $\tilde{K}$ has nonconstant global regular functions,
  which is a contradiction.
\end{proof}

\begin{proof}[Proof of \cref{stabs-are-flats-compact}]
Suppose that $x \in X$ is a distinguished point with stabilizer $V_x$ 
  of codimension at least two in $V$. 
  We wish to show that $V_x$ is the intersection of the codimension one 
  stabilizers containing it, so by induction it is enough to find 
  $y,z \in X$ such that 
  \[
    V_x = V_y \cap V_z,\quad \dim V_y = \dim V_z = \dim V_x + 1.
  \]
  By \cref{orbit-flat}, $V \cdot y \not= V \cdot z$ 
  implies $V_y \not= V_z$. Therefore it suffices to show that the orbit 
  closure $\overline{V \cdot x}$ contains two distinct orbits of 
  codimension one in $\overline{V \cdot x}$. Suppose that $V \cdot y \subseteq 
  \overline{V \cdot x}$ is an orbit of codimension one in 
  $\overline{V \cdot x}$. Since $V_x \subseteq V$ is codimension at least 
  two, $\dim V \cdot y > 0$, and so we can choose a nonconstant regular function 
  $f$ on $V \cdot y$. By \cref{extending-functions}, $f$ extends to a 
  regular function on the minimal invariant neighborhood $U_y \supseteq 
  V \cdot x \cup V \cdot y$, which is nonconstant when restricted to $V 
  \cdot x \cup V \cdot y$. Therefore by 
  \cref{codim-one-boundary} with $U = V \cdot x \cup V \cdot y$ and $K = 
  \overline{V \cdot x}$, there is another orbit $V \cdot z \subseteq 
  \overline{V \cdot x}$ of codimension one.
\end{proof}
\begin{proof}[Proof of \cref{stabs-are-flats-noncompact}]
  We will apply 
\cref{stabs-are-flats-compact} to the slices of $X$. 
  We have that $\{0\} \in \LL(X)$, and by 
  \cref{closed-under-intersection}, $\LL(X)$ is closed under 
  intersections. It remains to show that for $F \in \LL(X)$, 
  $
    \{G \in \LL(X) : G \subseteq F\}
  $
  is the collection of flats of a partial 
  hyperplane arrangement in $F$. Suppose that $F$ is the stabilizer of 
  the distinguished point $x \in X$. Then the slice $Z_x$ is a proper 
  linear $F$-variety by \cref{closed-under-slices} and 
  \cref{prop:proper-slice}, so the set of stabilizer $\LL(Z_x)$ is the collection of flats of an 
  essential hyperplane arrangement in $F$ by \cref{stabs-are-flats-compact}. 
  Then by \cref{orbits-and-stabs-of-GHZ},
  $
    \LL(Z_x) = \{G \in \LL(X) : G \subseteq F\}.
  $
\end{proof}

\section{Proof of \cref{main-thm-compact}} \label{sec:structure-of-linear-compactifications}

  Let $X$ be an equivariant compactification of $V$, such that $X$ is 
  normal and satisfies \FO{} and \SL{}. We have shown in 
  \cref{slices-to-1psg} and \cref{1psg-to-slices} that this is 
  equivalent to assuming $X$ is normal and satisfies \FO{} and \OP{}. In 
  \cref{coord-interpretations} we show that matroid Schubert varieties 
  satisfy \FO{} and \SL{}, and normality of matroid Schubert varieties 
  follows from \cite[Theorem 1]{B03} together with \cite[Theorem 
  1.3(c)]{AB16}. Thus it only remains to show that $X$ is equivariantly 
  isomorphic to a matroid Schubert variety associated to a hyperplane 
  arrangement in $V$.
  
  By \cref{prop:stabs-are-flats}, there exists an essential hyperplane arrangement $\A = \{H_1,\ldots, H_n\}$
  in $V$ whose lattice of
  flats is the collection of stabilizers of $X$. 
  We write
  \[
  \Phi_{\A}:V \to V/H_1 \times \ldots \times V/H_n,
  \]
  for the induced linear embedding, and $Y_\A$ for the Schubert
  variety of $\A$.
  Our goal is to show that there exists an isomorphism
  \[
  T:X \to Y_{\A(X)}
  \]
  extending the isomorphism
  \[
  \Phi_{\A}:V \to \Phi_{\A}(V).
  \]
  For each hyperplane $H_i$, we denote by $x_i, Z_i,$ and $V *_{H_i} Z_i \subseteq X$ the corresponding distinguished point, slice, and minimal $V$-invariant open neighborhood, respectively. Explicitly, $Z_i$
  is the closure of $H_i$ in $X$, $x_i$ is the $H_i$-fixed point in
  $Z_i$, and $V *_{H_i} Z_i$ is embedded in $X$ as the union of all $V$-orbits in $X$ which
  intersect $Z_i$ (see \cref{compactifications-with-slices}). Because $Z_i$ is
  proper (\cref{prop:proper-slice}) and $X$ is separated, $Z_i \subseteq X$
  is closed. Recall from \cref{sec:slice-def} that $Z_i$ is the fiber of the trivial $V$-equivariant fibration
  \[
  \tau:V *_{H_i} Z_i \to V/H_i.
  \]
  Therefore $Z_i$ is a prime Cartier divisor in $X$, so there is an
  associated line bundle $\OO_X(Z_i)$. Fix a
  linearization of $\OO_X(Z_i)$, which exists by \cite[Theorem 7.2]{D03}. We then have a linear action
    \[
    V \curvearrowright H^0(X, \OO_X(Z_i)).
    \]
    Because $\tau$ is $V$-equivariant, the translations of $Z_i$ under the action of $V$ are the fibers of $\tau$, so letting $Z_i' \not= Z_i$ be any such translation, we have that $Z_i$ and $Z_i'$ are linearly equivalent and disjoint.
    Thus $\OO_X(Z_i)$ is globally generated, since the sections (up to
  scaling) corresponding to $Z_i$ and $Z_i'$ have no common
  zeros. So far, we have that $Z_i$ defines a $V$-equivariant morphism
    \[
    T_i:X \to \P(H^0(X, \OO_X(Z_i))^\vee).
    \]
Finally, we have that the target of $T_i$ is $\P^1$ from a general observation:
    \begin{lem}
      Suppose $X$ is a proper normal variety, and $Z$ and $Z'$ are
      prime Cartier divisors which are linearly equivalent and such
      that $Z \cap Z' = \emptyset$. Then the space of global sections
      of $\OO_X(Z)$ is two dimensional.
    \end{lem}
    \begin{proof}
      Let $i:Z \to X$ denote the inclusion, and consider the short exact sequence
      \begin{center}
        \begin{tikzcd}
0 \arrow[r] & \mathcal{O}_X \arrow[r] & \mathcal{O}_X(Z) \arrow[r] & \mathcal{O}_X(Z) \otimes i_*\mathcal{O}_{Z} \arrow[r] & 0.
        \end{tikzcd}
      \end{center}
      By the projection formula, the sheaf on the right is isomorphic to $i_*(i^*\OO_X(Z) \otimes \OO_Z)$. However the restriction of $\OO_X(Z)$ to $Z$ is trivial because $Z$ can be moved to the disjoint divisor $Z'$. Thus 
    $
    \OO_X(Z) \otimes i_*\OO_{Z} \cong i_* \OO_Z.
    $
    Now take the long exact sequence in cohomology.
    \begin{center}
\begin{tikzcd}
0 \arrow[r] & {H^0(X, \mathcal{O}_X)} \arrow[r] & {H^0(X, \mathcal{O}_X(Z))} \arrow[r] & {H^0(X, i_*\mathcal{O}_Z)} \arrow[r] & \ldots
\end{tikzcd}
      \end{center}
    Since $X$ and $Z$ are proper and irreducible, $\dim H^0(X, \OO_X) =
    \dim H^0(X, i_*\OO_Z) = 1$. Therefore $\dim H^0(X, \OO_X(Z)) \leq
    2$. We also have that the sections (up to scaling) corresponding to $Z$ and $Z'$
    are independent, so $\dim H^0(X, \OO_X(Z)) = 2$.
    \end{proof}
    Let us choose coordinates on the target of $T_i$:
    \[
    s_0,s_1 \in H^0(X, \OO_X(Z_i)),\quad \Div(s_0) = Z_i,\quad s_1
    \not=0\ \text{is $V$-fixed}.
    \]
    The section $s_1$ exists because $V$ is unipotent. For any
    isomorphism between $\OO_X(Z_i) \vert_{V}$ and $\OO_V$, we have that
    $s_0\vert_V$ is sent to a linear form vanishing on $H_i$, and $s_1
    \vert_V$ is sent to a constant since $s_0$ is $V$-fixed. Thus
    there is a commutative square
    \begin{center}
        \begin{tikzcd}
        X \arrow[r, "T_i"]        & \mathbb{P}(V/H_i \oplus \C)          \\
        V \arrow[r] \arrow[u, hook] & V/H_i \arrow[u, hook]
        \end{tikzcd}
    \end{center}
    where the right vertical arrow is the embedding
    \[
    V/H_i \to \P^1,\quad v \mapsto [s_0(v):s_1(v)].
    \]
    From this it follows
    that the product map $X \to \prod_{i = 1}^n \P(V/H_i \oplus \C)$
    extends $\Phi_\A$, and thus
    we can define a morphism
    \[
    T:X \to Y_{\A(X)}, \quad T := (T_1,\ldots,T_n).
    \]
    Since $T$ is birational, by \cref{thm:open-criterion}
    we can show that $T$ is an isomorphism by showing that it is
    bijective on closed points. Since $T$ extends $\Phi_{\A}$, it is a 
    morphism of linear $V$-varieties. The set of stabilizers 
    of $X$ is the lattice of flats of $\A$ by assumption, and one can prove in coordinates that the set
    of stabilizers of $Y_\A$ is the image under $\Phi_\A$ of the
    lattice of
    flats of $\A$ (see \cref{coord-stab}). Thus $T$ carries the 
    set of stabilizers of $X$ bijectively onto the set of 
    stabilizers of $Y_\A$. Furthermore, $T$ carries the distinguished 
    points of $X$ bijectively onto the distinguished points of $Y_\A$
    by \cref{orbit-flat}.
    Let $x \in X$ be a distinguished point with stabilizer $V_x$, and let $T(x) \in Y_\A$ be the corresponding distinguished point with stabilizer $\Phi_\A(V_x)$. We have the
    following commutative square relating the orbit $V\cdot x$ in $X$ and
    the corresponding orbit $\Phi_{\A}(V) \cdot T(x)$ in $Y_{\A}$.
    \begin{center}

      \begin{tikzcd}
V\cdot x \arrow[r, "T"] \arrow[d, "\cong"] & \Phi_{\A}(V)\cdot T(x) \arrow[d, "\cong"]               \\
V/V_x \arrow[r] & \Phi_{\A}(V)/\Phi_{\A}(V_x)
\end{tikzcd}
      \end{center}
    Since the bottom arrow is an isomorphism, it follows that the top arrow is an
    isomorphism, and we take the disjoint union over all distinguished points to
    obtain that $T$ is a bijection on closed points.

    This completes the proof of \cref{main-thm-compact}.

\section{Proof of \cref{main-thm-noncompact}} \label{sec:structure-of-linear-V-varieties}

\subsection{Overview} In this section $V = \C^d$ will denote a vector group, and we will use
the term \emph{linear $V$-variety} to mean a normal equivariant
partial compactification $V \subseteq X$, which has finitely many
orbits and a slice through every point. We have proved in
\cref{slices-and-1psg} that
\cref{main-thm-noncompact-slices} and \cref{main-thm-noncompact-1psg}
in \cref{main-thm-noncompact} are equivalent. To prove \cref{main-thm-noncompact}, it remains to
show the following.

\begin{thm}
Given a linear $V$-variety $X$, let $\LL(X)$ denote the associated 
  partial hyperplane arrangement of \cref{def:L-of-X}.
\begin{thmenum}[leftmargin=*]
  \item \label{functoriality} \textbf{Functoriality:} If $X_i$ is a linear $V_i$-variety for $i = 1,2$ and $T:X_1 \to X_2$ is a morphism, then the restricted linear map $T:V_1 \to V_2$ is a morphism of partial hyperplane arrangements $\LL(X_1) \to \LL(X_2)$.
  \item \label{full-faithfulness} \textbf{Full faithfulness:} If $X_i$ 
    is a linear $V_i$-variety for $i = 1,2$ and $T:V_1 \to V_2$ is a 
    morphism of partial hyperplane arrangements $\LL(X_1) \to \LL(X_2)$, 
    then $T$ extends uniquely to a morphism of linear $V$-varieties $X_1 
    \to X_2$.
    \item \label{essential-surjectivity} \textbf{Essential surjectivity:} If $\LL$ is an essential hyperplane arrangement in $V$, then there exists a linear $V$-variety $X$ such that $\LL(X) \cong \LL$.
\end{thmenum}
\end{thm}

The proof of essential surjectivity describes how to construct the
linear $V$-variety associated to a partial hyperplane arrangement.

\subsection{Morphisms}
In this section we prove functoriality and full faithfulness. To prove 
functoriality, we generalize the proof of 
\cref{closed-under-intersection}, appealing again to the Borel fixed 
point theorem.
\begin{proof}[Proof of \cref{functoriality}]
  Since
  the stabilizer of $x \in X_1$ is mapped into the stabilizer of $T(x)
  \in X_2$, it follows that for each flat of $\LL(X_1)$ is mapped into a 
  flat of $\LL(X_2)$, as required in \cref{flats-to-flats} of \cref{morphism-of-arrangements}.
  If $F_1 \in \LL(X_1)$ and $F_2 \in \LL(X_2)$, it remains to show
  that $T^{-1}(F_2) \cap F_1 \in \LL(X_1)$. Write $Z_1$ for the minimal
  slices through the distinguished point $x_1 \in X_1$ corresponding to 
  $F_1$. Then $Z_1$ is proper by \cref{prop:proper-slice}, 
  so $T^{-1}(F_2) \cap Z_1$ is proper.
  By the Borel fixed point theorem, there exists $z \in
  T^{-1}(x_2) \cap Z_1$ such that 
  \[
    T^{-1}(F_2) \cap F_1 \subseteq (V_1)_{z}.
  \]
  We now show the opposite inclusion.  By \cref{prop:proper-slice} there 
  is a unique $F_1$-fixed point in 
  $Z_1$, and $\overline{F_1 \cdot z}$ contains a $F_1$-fixed point, so 
  $x_1 \in \overline{F_1 \cdot z}$. Therefore $(V_1)_z \subseteq F_1$. 
  On the other hand, $z \in T^{-1}(x_2)$, so $(V_1)_z \subseteq 
  T^{-1}(F_2)$.
\end{proof}

Let us start by proving full faithfulness in the compact case.
\begin{lem}\label{full-faithfulness-compact}
  Suppose that $\A_i$ is an essential hyperplane arrangement in the 
  vector group $V_i$ for $i=1,2$, and $T:V_1 \to V_2$ is a linear map 
  such that the preimage of each hyperplane in $\A_2$ is either a hyperplane in $\A_1$ or is $V_1$.
  Then $T$ extends to a morphism between matroid Schubert varieties $Y_{\A_1} \to Y_{\A_2}$.
\end{lem}
\begin{proof}
  Since $Y_{\A_i}$ is the closure of $V_i \subseteq \prod_{H \in \A_i} 
  \P(V_i/H \oplus \C)$, we are reduced to showing that $T$ extends to a map 
  \[
  \prod_{H_1 \in \A_1}\P(V_1/H_1 \oplus \C) \to \prod_{H_2 \in \A_2} 
    \P(V_2/H_2 \oplus \C).
  \]
  Fix $H_2 \in \A_2$. If $T^{-1}(H_2) = V_1$, then the $H_2$ 
  component of the displayed function can be defined to be constant with 
  value $0 \in V_2/H_2$. Otherwise $T^{-1}(H_2) \in \A_1$, in which case the $H_2$ 
  component of the displayed function can be defined as projection onto 
  $\P(V_1/T^{-1}(H_2) \oplus \C)$ followed by the map
  \[
[v + T^{-1}(H_2): z] \mapsto [T(v) + H_2: z].
\]
\end{proof}

To prove full faithfulness in general, we apply 
\cref{full-faithfulness-compact} to the slices in a linear $V$-variety, 
which are matroid Schubert varieties by 
\cref{prop:proper-slice} and \cref{main-thm-compact}.
\begin{proof}[Proof of \cref{full-faithfulness}]
  Since $V_1 \subseteq X_1$ is dense, uniqueness follows immediately. We 
  now argue that an extension $T:X_1 \to X_2$ exists.
  
  Suppose that $F_i \in \LL(X_i)$ for $i=1,2$ such that $T(F_1) \subseteq F_2$, and 
  denote by $Z_i \subseteq X_i$ the slices through the corresponding 
  distinguished points. By \cref{prop:proper-slice},
  \cref{closed-under-slices}, and \cref{main-thm-compact}, $Z_i$ is
  the matroid Schubert variety associated to a hyperplane arrangement in 
  $F_i$. By
  \cref{orbits-and-stabs-of-GHZ}, the hyperplane arrangement corresponding to $Z_i$ is given by those flats of $\LL(X_i)$ which are contained in $F_i$.
  Therefore, since $T$ is a morphism of partial hyperplane arrangements,
  the hypotheses of \cref{full-faithfulness-compact} are satisfied for the restriction $T\vert_{F_1}:F_1 \to F_2$.
  Thus $T\vert_{F_1}$ extends to a morphism $\overline{T \vert_{F_1}}:Z_1 \to Z_2$.

  We can now extend $T$ to the open set $V_1 *_{F_1} Z_1 \subseteq X_1$ by sending $[v,z]$ to $[T(v), \overline{T \vert_{F_1}}(z)] \in V_2 *_{F_2}Z_2.$
  As $F_1$ and $F_2$ vary, the open sets $V_1 *_{F_1} 
  Z_1$ cover $X_1$ (here we are using \cref{flats-to-flats} in 
  \cref{morphism-of-arrangements}), and the extensions of $T$ to $V_1 *_{F_1} Z_1$ are unique and thus compatible on 
  intersections, so they define a global extension $X_1 \to X_2$.
\end{proof}

\subsection{Construction of linear $V$-varieties}
Now we turn to essential surjectivity. Let $\LL$ be a partial hyperplane arrangement in $V$. 

\subsubsection{A diagram of hyperplane arrangements}\label{A diagram of 
hyperplane arrangements} From
\cref{def:partial-hyperplane-arrangement} we have that for every $F
\in \LL$, there is an essential hyperplane arrangement $\A_F$ in the
vector space $F$ whose lattice of flats is $\{F' \in \LL : F' \subseteq
F\}$. It follows immediately that $\A_F$ is unique. If $F' \subseteq F$
are elements of $\LL$, then $\A_{F'}$ is called the \emph{restriction} of
$\A_F$ to the flat $F'$.

\subsubsection{A diagram of matroid Schubert
  varieties} \label{A diagram of matroid Schubert varieties}
For each $F \in \LL$, we have the matroid Schubert variety $Y_{\A_F}$
associated to the hyperplane arrangement $\A_F$. If $F'\in \LL$ is
contained in $F$, then $F'$ is a flat of $\A_F$. Therefore $Y_{\A_{F'}}$
is the slice through a distinguished point $x_F' \in Y_{\A_F}$ by the
coordinate formula given in \cref{coord-slice-dist-pt}. So $\LL$
indexes a diagram of matroid Schubert varieties,
where each arrow is the inclusion of a slice.

\subsubsection{A diagram of open embeddings}\label{A diagram of open embeddings}
Given $F' \subseteq F$ elements of $\LL$, we have an open embedding $F 
*_{F'} Y_{\A_{F'}} \subseteq Y_{\A_F}$ because $Y_{\A_{F'}}$ is a slice through 
$x_{F'}$. It is straight forward to check that $V *_F -$ preserves open 
embeddings, so by the associativity property of \cref{associativity}
we have an open embedding
\[
  V *_G Y_{\A_G} \cong V *_F (F *_G Y_{\A_G}) \subseteq V *_F Y_{\A_F}. 
\]
Therefore $\LL$ 
indexes a diagram of open embeddings between the varieties $V *_F 
Y_{\A_F}$ for $F \in \LL$. By \cref{orbits-and-stabs-of-GHZ}, $V *_F Y_{\A_F}$ 
has finitely many $V$-orbits, and again by the associativity property of
\cref{associativity} each point
$[v,y] \in V *_F Y_{\A_F}$ has a slice $V *_F Z_y$, where $Z_y$ is a 
slice through $y \in Y_{\A_F}$. Thus $V *_F Y_{\A}$ is a linear 
$V$-variety.

\subsubsection{Cocycle condition}
To verify that the $V *_F Y_{\A_F}$ glue together, we must check the cocycle condition \cite[Exercise II.2.12]{H77}. This reduces to the following lemma.

\begin{lem}
  If $F',F'' \subseteq F$ are elements of $\LL$, then
  \[
    V *_{F'} Y_{\A_{F'}} \cap V *_{F''} Y_{\A_{F''}} = V *_{F' \cap F''} 
    Y_{\A_{F' \cap F''}}
  \]
  considered as open sets in $V *_F Y_{\A_F}$.
\end{lem}
\begin{proof}
  A point $[v,z] \in V *_{F} Y_{\A_F}$ lies in $V *_{F'} 
  Y_{\A_{F'}}$ if and only if $z \in Y_{\A_{F'}}$, so we are reduced to showing that 
  $
  Y_{\A_{F'}} \cap Y_{\A_{F''}} = Y_{\A_{F'\cap F''}},
  $
  considered as closed sets in $Y_{\A_F}$. To check this one can use the set 
  theoretic formula of \cref{set-theoretic-description}.
\end{proof}

\subsubsection{Separation}
We now prove that the variety $Y_\LL$ glued from the $V *_F Y_{\A_F}$ is 
separated. By \cite[Corralary II.4.2]{H77}, checking that the diagonal morphism $Y_\LL \to Y_\LL \times 
Y_\LL$ is a closed embedding reduces to the following lemma.

\begin{lem}
  Suppose that $F = F' \cap F''$ are elements of $\LL$, and write
  \[
    i:V *_F Y_{\A_F} \to V *_{F'} Y_{\A_{F'}} \times V *_{F''} 
    Y_{\A_{F''}},\quad 
    [v,y] \mapsto [v,y] \times [v,y].
  \]
  Then $i$ has a closed image.
\end{lem}
\begin{proof}
  Consider the canonical fibration defined in \cref{sec:slice-def},
  \[
  \tau_F:V *_F Y_{\A_F} \to V/F,\quad [v,y] \mapsto v+F.
  \]
  We get that the following square is commutative.
  \begin{center}
    \begin{tikzcd}
    V *_F Y_{\A_F} \arrow[d, "\tau_F"] \arrow[r, "i"] \arrow[rd, "f"] & 
    V *_{F'} Y_{\A_{F'}} \times V *_{F''} Y_{\A_{F''}} \arrow[d, "\tau_{F'} \times \tau_{F''}"] \\
    V/F \arrow[r, "j"]                                    & V/F' \times 
    V/F''                                           
    \end{tikzcd}
  \end{center}
    Notice that $j$ is a closed embedding, since $F = F' \cap F''$. We
    also have that $\tau_F$ is proper (it's conjugate to the projection
    $V/F \times Y_{\A_F} \to V/F$), so $f$ is proper. Then $i$
    is proper by \cite[Corralary II.4.8(e)]{H77}, and thus closed.
\end{proof}

\subsubsection{Linearity}
Since the action of $V$ extends to each open set in a cover of
$Y_{\LL}$ and is compatible on intersections, the action extends to
$Y_{\LL}$. From the fact that $Y_{\LL}$ is glued from linear
$V$-varieties with maps that preserve $V$, it follows that
$Y_{\LL}$ is normal, that it has finitely many $V$-orbits, and that
every point has a slice. Thus $Y_\LL$ is a linear $V$-variety.

\begin{proof}[Proof of \cref{essential-surjectivity}]
  We have constructed a linear $V$-variety $Y_\LL$, and we wish to show that the collection of stabilizers $\LL(Y_\LL)$ coincides with $\LL$.

  First we check that each flat of $\LL$ is the stabilizer of a point in $Y_\LL$. Suppose that $F \in \LL$. Then consider the point $[0,x_F] \in V *_F Y_{\A_F} \subseteq Y_\LL$, where $x_F$ is the unique fixed point of $Y_{\A_F}$. 
  By \cref{orbits-and-stabs-of-GHZ}, $F$ 
  is the stabilizer of $[0,x_F] \in Y_\LL$. 

  Now we check that the stabilizer of each point of $Y_\LL$ is a flat
  of $\LL$. From the construction of $Y_\LL$, we have that every 
  point is contained in an open set isomorphic to $V *_F Y_{\A_F}$ for 
  $F \in \LL$. Suppose that $[v,y] \in V *_F Y_{\A_F}$. Then the 
  stabilizer of $[v,y]$ with respect to the action of $V$ is equal the 
  stabilizer of $y \in Y_{\A_F}$ with respect to the action of $F$ by 
  \cref{orbits-and-stabs-of-GHZ}, and is therefore equal to a flat in 
  $\A_F$ by \cref{coord-stab}.
\end{proof}
This completes the proof of \cref{main-thm-noncompact}.

\subsubsection{Comparison with toric varieties}\label{construction-comparison}
We conclude this section by explaining how the construction of a
toric variety form a polyhedral fan can be made to look like
the construction above. In order to be consistent with \cite{CLS11},
we will use $n$ for the dimension of a toric variety rather than $d$
as we have been doing so far. We will use $d$ for the dimension of a cone. Let us fix the following notation.
\begin{itemize}[leftmargin=*]
  \item $N \cong \Z^n$, a lattice of dimension $n$,
  \item $T = \Spec \C[N^\vee]$, the corresponding $n$-dimensional torus,
  \item $\Sigma,$ a fan of strictly convex rational polyhedral cones in $N \otimes_\Z\R$,
  \item $\sigma \in \Sigma$, a cone of dimension $d$,
  \item $U_{\sigma,N},$ the toric variety of dimension $n$
    corresponding to $\sigma$ considered as a cone in $N \otimes_\Z
    \R$
    \cite[Theorem 1.2.18]{CLS11},
  \item $x_\sigma \in U_{\sigma,N},$ the distinguished point in the 
    minimal $T$-orbit \cite[Chapter 3.2]{CLS11},
  \item $N_\sigma = \Span_{\Z}(\sigma \cap N),$ the sublattice of dimension $d$ generated by $\sigma$,
  \item $T_{\sigma} = \Spec \C[N_\sigma^\vee]$, the $d$-dimensional torus corresponding to $N_\sigma$,
  \item $U_{\sigma,N_\sigma},$ the toric variety of dimension $d$ corresponding to $\sigma$ considered as a cone in $N_\sigma \otimes_\Z \R$,
  \item $\mathscr{H}_\sigma,$ the unique minimal basis (see \cite[Proposition 1.2.23]{CLS11}) for the semigroup
    \[
    S_{\sigma,N_\sigma} = \{u \in \Hom_{\Z}(N_\sigma, \Z) : u\ \text{is nonnegative on $\sigma$}\}.
    \]
    
\end{itemize}
The fan $\Sigma$ indexes a commutative diagram of inclusions among its cones, as in \cref{A diagram of hyperplane arrangements}.
The cone $\sigma$ defines an embedding of the torus
$T_{\sigma}$ in the larger torus $(\C^\times)^{\mathscr{H}_\sigma}$ (see 
\cite[Definition 1.1.7]{CLS11}),
and the closure of $T_\sigma \subseteq
(\mathbb{A}^1)^{\mathscr{H}_\sigma}$ is $U_{\sigma,N_\sigma}$ (see \cite[Theorem 1.1.17]{CLS11}), similar
to the definition of a matroid Schubert variety. Given $\tau \subseteq \sigma$ elements of $\Sigma$, the
natural homomorphism of tori $T_\tau \subseteq T_\sigma$ extends to a
morphism of toric varieties $U_{\tau, N_\tau} \subseteq U_{\sigma,
  N_\sigma}$, and one can check that $U_{\tau, N_\tau}$ is the minimal slice through
the distinguished point $x_\tau \in U_{\sigma, N_\sigma}$, as in
\cref{A diagram of matroid Schubert varieties}.
As in \cref{A diagram of open embeddings}, for elements $\tau \subseteq \sigma$ of $\Sigma$, we have a natural open embedding
\[
T *_{T_\tau} U_{\tau, N_\tau} \cong T *_{T_\sigma} (T *_{T_\tau} U_{\tau, N_\tau}) \subseteq T *_{T_\sigma} U_{\sigma, N_\sigma}.
\]
One can verify that $U_{\sigma, N_\sigma}$ is the minimal slice through the distinguished point $x_\sigma \in U_{\sigma, N}$, so there is natural isomorphism $T *_{T_\sigma} U_{\sigma, N_\sigma} \cong U_{\sigma, N}$. Thus the above diagram of open embeddings is isomorphic to the usual diagram of open embeddings $U_{\tau, N} \subseteq U_{\sigma, N}$ (see \cite[Section 3.1]{CLS11}), whose colimit is the toric variety $X_\Sigma$.

\appendix
\section{Matroid Schubert varieties in 
coordinates}\label{appendix}

In this appendix we give a coordinate formula for the matroid Schubert 
variety associated to a hyperplane arrangement as a closed subset of 
$(\P^1)^n$, as well as 
coordinate formulas for the orbits, distinguished points, stabilizers, 
minimal $V$-invariant open neighborhoods, and slices.
All of the following formulas are consequences of the defining 
multihomogeneous equations of the matroid Schubert variety given in \cite{AB16},
and they will be familiar to the experts.

We fix the following notation.
\begin{itemize}[leftmargin=*]
  \item \textbf{Coordinates:} Set $E = \{1,\ldots,n\}$, and let $\A = \{H_1, \ldots, H_n\}$ be an essential hyperplane arrangement in $V$. In order to work with coordinates, let us fix an isomorphism of $V/H_i$ with $\C$ for each $i$, and thus we can consider $V \subseteq \C^E$.
\item \textbf{Group action:} Let us identify $\P^1 = \C \cup \{\infty\}$ set theoretically, writing $z$ for $[z:1]$ and $\infty$ for $[1:0]$. In this notation, the action of $\C^E$ on $(\P^1)^E$ is given by coordinate wise addition, using the rule $z + \infty = \infty$ for all $z \in \C$. Since the action of $V$ on $Y_\A$ is restricted from the action of $V$ on $(\P^1)^E$, we have that the action of $V$ on $Y_\A$ is also given by coordinate wise addition.

\item \textbf{Projections:} Given $S \subseteq E$, write $\pi_S:(\P^1)^{E} \to (\P^1)^{S}$ for the coordinate projection. Because we consider $\C^E \subseteq (\P^1)^E$, we will also write $\pi_S(V)$ for the coordinate projection of $V$ onto $\C^S \times \{0\}^{E \setminus S}$.

\item \textbf{Matroid flats:} A \emph{flat} of the matroid associated
  to $V \subseteq \C^E$ is a subset $F \subseteq E$ such that $F = \{i
  \in E : v_i = 0\}$ for some $v \in V$. Write $\mathscr{F}$ for the
  collection of flats of the matroid associated to $V$. There is a
  natural bijection between $\mathscr{F}$ and the lattice of flats of
  $\A$ given by sending $F \in \mathscr{F}$ to $\cap_{i \in F}H_i$.

\end{itemize}

We begin with a set theoretic description of $Y_\A$:
\begin{prop}[\cite{PXY18} Lemma 7.5 and Lemma 7.6] \label{set-theoretic-description}
  Write $Y_\A$ as the closure of the linear subspace $V \subseteq \C^E$. Fix a point $x \in (\P^1)^E$, and write $F \subseteq E$ for the set of indices corresponding to non-infinite entries of $x$. Then $x \in Y_\A$ if and only if $F \in \mathscr{F}$ and $\pi_F(x) \in \pi_F(V)$. Equivalently,
  \[
  Y_\A = \bigcup_{F \in \mathscr{F}} \pi_F(V) \times \{\infty\}^{E \setminus F} \subseteq (\P^1)^E.
  \]
\end{prop}

Let us now demonstrate in explicit coordinates the objects 
defined in \cref{compactifications-with-slices}.

\begin{cor}\label{coord-interpretations}
  Let $x,y \in Y_\A$, and write $F,F' \subseteq E$ for the set of indices corresponding to non-infinite entries of $x$ and $y$ respectively.
  \begin{corenum}
  \item \label{coord-orbit} The $V$-orbit of $x$ is $V \cdot x = \pi_F(V) \times \{\infty\}^{E \setminus F}$.
  \item \label{coord-dist-pt} The distinguished point in the $V$-orbit of $x$ is
      $x_F = \{0\}^F \times \{\infty\}^{E \setminus F}$.
  \item \label{coord-stab} The stabilizer of $x$ is $V_x = V \cap (\{0\}^F \times \C^{E \setminus F})$.
  \item \label{coord-ngbd} The minimal $V$-invariant open neighborhood 
    $U_x$ of $x$ contains $y$ if and only if $F \subseteq F'$. Equivalently,
    \[
      U_x = \bigcup_{F' \in \mathscr{F}, F \subseteq F'} \pi_{F'}(V) \times 
    \{\infty\}^{E \setminus F'}.
    \]
  \item \label{coord-slice} The minimal slice $Z_x$ through $x$ contains 
    $y$ if and only if $F \subseteq F'$ and $\pi_F(x) = \pi_F(y)$. Equivalently
    \[
      Z_x = \bigcup_{F' \in \mathscr{F}, F \subseteq F'} \left(\pi_{F'}(V) \cap 
    \left(\pi_F(x) \times \C^{F' \setminus F}\right)\right) \times \{\infty\}^{E \setminus 
    F'}.
    \]
  \item \label{coord-slice-dist-pt} Set $Y_{\A_F}$ equal to the closure 
    of $V \cap \C^{E \setminus F}$ in $(\P^1)^{E \setminus F}$. The 
    minimal slice through the distinguish point $x_F = \{0\}^F \times 
    \{\infty\}^{E \setminus F'}$ is
    $
    Z_{x_F} = \{0\}^F \times Y_{\A_F}.
    $
\end{corenum}
\end{cor}
\begin{proof}
  We begin with \cref{coord-orbit}. Since $x_i = \infty$ for $i \not\in F$, we have that $v \in V$ acts on $x$ via $(v\cdot x)_i = v_i + x_i$ for $i \in F$ and $(v \cdot x)_i = \infty$ for $i \not\in F$. Thus $V \cdot x \subseteq \pi_F(V) \times \{\infty\}^{E \setminus F}$. For the reverse inclusion, let $y \in \pi_F(V) \times \{\infty\}^{E\setminus F}$, and choose $v \in V$ such that $\pi_F(v) = \pi_F(y) - \pi_F(x)$. Then $v \cdot x = y$.

  We have that $v \cdot x = x$ if and only $v_i + x_i = x_i$ for all $i \in F$, so \cref{coord-stab} follows.

  To prove \cref{coord-ngbd}, let $y \in Y_\A$ and write $F' = \{i \in E 
  : y_i \not= \infty\}$. We wish to show that the set of $y$ for which 
  $F \subseteq F'$ is equal to the minimal open neighborhood $U_x$. We 
  first note that the set of $y$ such
  that $F \subseteq F'$ is a $V$-stable open set of $V \cdot
  x$, so it contains the minimal one. To show the reverse inclusion we 
  must check that if $F \subseteq F'$, then $\pi_F(V) \times 
  \{\infty\}^{E \setminus F} \subseteq \overline{\pi_{F'}(V) \times 
  \{\infty\}^{E \setminus F'}}$. Since $F \in \mathscr{F}$ is a flat, we
  may choose a vector $v \in V$ such that $v_i = 0$ for $i \in F$ and
  $v_i \not= 0$ for $i \not\in F$. Then for each value of $t \in \C$, 
  $\pi_{F'}(tv) \times \{\infty\}^{E \setminus F'} \in \pi_{F'}(V) \times 
  \{\infty\}^{E \setminus F'}$, but as $t \to \infty$, the limit lies in $\pi_F(V) \times \{\infty\}^{E \setminus F}$.

  We now turn to \cref{coord-slice}. We have that
  \[
  Z_x := \{y \in Y_\A : F \subseteq F',\ \pi_F(x) = \pi_F(y)\}
  \]
  is contained in $U_x$ by \cref{coord-ngbd}, and so we just need to
  check that it is a slice. We have that $V_x$ acts on $Z_x$, so we use 
  \cref{injectivity-criterion} to show that the induced map $V *_{V_x}Z_x\to Y_\A$ 
  is injective on closed points. Suppose that $v
  \cdot z = v' \cdot z'$ for $v,v' \in V$ and $z,z' \in Z_x$. Then $\pi_F(v) + \pi_F(z) = \pi_F(v') + \pi_F(z')$, however $\pi_F(z) = \pi_F(z') = \pi_F(x)$, so $\pi_F(v - v') = 0$ as required. Since $V \cap Z_x$ is a coset of $V_x$, we have that $V *_{V_x}Z_x \to Y_\A$ is birational, and thus an open embedding by \cref{thm:open-criterion}.

  Now \cref{coord-dist-pt} and \cref{coord-slice-dist-pt} follow from knowing the slice through $x_F$.
\end{proof}

\printbibliography
\end{document}